\documentclass[11pt,reqno]{amsart}
\usepackage[margin=1in]{geometry}
\usepackage{amsmath,amssymb,amsthm}
\usepackage[hidelinks]{hyperref}
\usepackage{thmtools}
\usepackage[capitalize]{cleveref}
\usepackage[noadjust]{cite}
\usepackage{xcolor}

\newtheorem{theorem}{Theorem}[section]
\newtheorem{proposition}[theorem]{Proposition}

\newtheorem{question}[theorem]{Question}
\theoremstyle{remark}

\numberwithin{equation}{section}

\newcommand{\F}{\mathbb F}
\newcommand{\E}{\mathbb E}
\newcommand{\X}{\mathbf X}
\newcommand{\Y}{\mathbf Y}

\newcommand{\paren}[1]{\left(#1\right)}
\newcommand{\sqb}[1]{\left[#1\right]}
\DeclareMathOperator{\ex}{ex}

\title{$K_{2,t+1}$-free graphs with many copies of $K_{t,t}$}
\author{Cosmin Pohoata}
\author{Jonathan Tidor}
\author{Hung-Hsun Hans Yu}

\thanks{Pohoata was supported by NSF grant DMS-2246659. Yu was supported by a Jane Street Graduate Fellowship.}

\address{Department of Mathematics, Emory University, Atlanta, GA 30322, USA}
\email{cosmin.pohoata@emory.edu}

\address{Department of Mathematics, Princeton University, Princeton, NJ 08544, USA}
\email{\{jtidor,hansonyu\}@princeton.edu}

\date{}

\begin{document}

\begin{abstract}
For every fixed integer $t\geq 3$, we construct an $n$-vertex $K_{2,t+1}$-free graph containing $\Omega_t(n^2)$ copies of $K_{t,t}$. Combined with a simple counting argument, this shows that
\[
\ex(n,K_{t,t},K_{2,t+1})=\Theta_t(n^2).
\]
This answers a question of Spiro.
\end{abstract}

\maketitle

\section{Introduction}

For graphs $H$ and $F$, let $\ex(n,H,F)$ denote the maximum number of copies of $H$ in an $n$-vertex graph containing no copy of $F$. This generalization of the classical Tur\'an number was first systematically studied by Alon and Shikhelman~\cite{AS16}.

Answering a question of Spiro, we determine the following generalized Tur\'an number.

\begin{theorem}\label{thm:main}
For every fixed integer $t\ge 3$,
\[
\ex(n,K_{t,t},K_{2,t+1})=\Theta_t(n^2).
\]
In particular, there exists a constant $c_t>0$ such that for every sufficiently large integer $n$ there is a $K_{2,t+1}$-free graph on $n$ vertices containing at least $c_t n^2$ copies of $K_{t,t}$.
\end{theorem}

We note that there are quite a few works studying $\ex(n,K_{a,b},K_{c,d})$ -- the generalized Tur\'an problem for complete bipartite vs. complete bipartite -- for various values of $a,b,c,d$. See, e.g., the survey \cite[Section 4.3]{GP25} and the references contained therein. However, to the best of our knowledge, this is the first nontrivial result in the regime where $c<a\leq b<d$.

The upper bound in \cref{thm:main} is elementary. Indeed, let $G$ be an $n$-vertex $K_{2,t+1}$-free graph. For distinct vertices $v_1,v_2$, any copy of $K_{t,t}$ on vertex set $V\sqcup W$ with $v_1,v_2\in V$ must have $W$ contained in the common neighborhood $N(v_1)\cap N(v_2)$. By the $K_{2,t+1}$-free hypothesis, this common neighborhood has size at most $t$. Hence there is at most one choice for $W$. Once $W$ is fixed, note that $V$ must lie in
$N(W)\subseteq N(w_1)\cap N(w_2)$ for any distinct vertices $w_1,w_2\in W$. Again this is a set of size at most $t$, so there is at most one choice for $V$. Thus there is at most one possible $K_{t,t}$ once $v_1,v_2$ are chosen. Since there are $\binom n2$ choices for $v_1,v_2$, we obtain
\[
\ex(n,K_{t,t},K_{2,t+1})=O(n^2).
\]
Thus the content of \cref{thm:main} is a matching lower bound. 

Our construction is geometric. We let $G=(X\sqcup Y,E)$ be a point-plane incidence graph in $\F_q^3$ where $X\subseteq\F_q^3$ is a set of points so that each line contains at most $t$ points and a positive proportion of lines contain exactly $t$ points. $Y$ is a set of planes with the dual property: each line is contained in at most $t$ planes and a positive proportion of lines are contained in exactly $t$ planes. Such an incidence graph is always $K_{2,t+1}$-free; we will show later how to construct $X,Y$ so that the graph has many copies of $K_{t,t}$.

In some sense, the idea of looking a point-plane incidence graph in $\F_q^3$ is inspired by the recent work of Milojevi\'c, Sudakov, and Tomon~\cite{MST24}. In that work the authors use the fact that in such a graph, any copy of $K_{3,3}\setminus\{e\}$ can be completed to a copy of $K_{3,3}$; a similar property is crucial for our construction so that the elementary upper bound argument gives an asymptotically tight estimate on the number of $K_{t,t}$'s.

The sets $X,Y$ used in our construction are subspace evasive sets, first studied by Pudl\'ak and R\"odl~\cite{PR04}. In particular, these sets are $(1,t)$-subspace evasive sets in $\F_q^3$, meaning that any line (i.e., 1-dimensional affine subspace) contains at most $t$ points of the set. Large subspace evasive sets were first constructed explicitly by Dvir and Lovett~\cite{DL12}. In \cite{ST24}, Sudakov and Tomon then also showed that one can construct large subspace evasive sets using the randomized algebraic method, in the spirit of Bukh~\cite{Buk15} (as pointed out in \cite{ST24}, Conlon had also observed this independently). We will show that a slight modification of this randomized algebraic construction has the additional property that we need: for each line, the probability that it contains exactly $t$ points of the set is bounded away from zero.

Finally, to put our result in perspective, we note that F\"uredi gave a construction of an $n$-vertex $K_{2,t+1}$-free graph with $(\sqrt t/{2}+o(1))n^{3/2}$ edges~\cite{Fur96}. By the K\H{o}v\'ari--S\'os--Tur\'an theorem, this is the maximum possible number of edges~\cite{KST54}. However, we point out that this construction does not have many copies of $K_{t,t}$; in fact we will show in the appendix that it is $K_{3,t}$-free.

\subsection*{Acknowledgments}
We would like to thank Sam Spiro for bringing this problem to our attention.

\section{Incidence graphs from subspace evasive sets}
\label{sec:construction}

A set $X\subseteq\F_q^n$ is called $(k,c)$-subspace evasive if every $k$-dimensional affine subspace of $\F_q^n$ contains at most $c$ points of $X$. We first describe our main tool which constructs $(1,t)$-subspace evasive sets in $\F_q^3$ where each line has exactly $t$ points with positive probability.

\begin{proposition}
\label{prop:subspace-evasive}
Given $t\geq 3$ and a prime power $q$, there exists a random set $\X \subseteq \F_q^3$ with the following property. For each line $\ell\subset\F_q^3$, we have $|\X\cap \ell|\leq t$ and
\[\Pr(|\X\cap\ell|=t)=\frac1{t!}-o_{t;q\to\infty}(1).\]
\end{proposition}

Using this result, we can deduce the main theorem.

\begin{proof}[Proof of \cref{thm:main}]
Let $\X,\Y\subseteq\F_q^3$ be independent copies of the random set produced by \cref{prop:subspace-evasive}. Let $G=(\X\sqcup\Y,E)$ be the bipartite graph with vertex sets $\X,\Y$ where $x\in\X\subseteq\F_q^3$ is adjacent to $y\in\Y\subseteq\F_q^3$ if $x_1y_1+x_2y_2+x_3y_3=1$.

Since we can partition $\F_q^3$ into $q^2$ lines, we see that $|\X|,|\Y|\leq tq^2$. Thus $G$ is a graph on $n=|\X|+|\Y|\leq 2tq^2$ vertices.

First we show that $G$ is $K_{2,t+1}$-free. For a point $x\in\F_q^3$, let $H_x=\{y\in\F_q^3:x_1y_1+x_2y_2+x_3y_3=1\}$. This is a plane (unless $x=0$ in which case it is the empty set). Now suppose $x_1,x_2\in\X$ and $y_1,\ldots,y_{t+1}\in\Y$ form a copy of $K_{2,t+1}$ in $G$. By definition, this means that $y_1,\ldots,y_{t+1}\in H_{x_1}\cap H_{x_2}$. The latter set is a line or empty, so by definition it contains at most $t$ points of $\Y$, contradiction. The same argument also shows that $G$ contains no $K_{2,t+1}$ with $t+1$ vertices on the left and 2 vertices on the right.

Next we show that $G$ contains many copies of $K_{t,t}$. For each line $\ell\subset\F_q^3$, let $\ell^*=\{z\in\F_q^3:\ell\subset H_z\}$ be the dual line. If $\ell$ does not pass through the origin, then $\ell^*$ is a line in $\F_q^3$ that does not pass through the origin. To see this, write $\ell=x+\F_q y$. Then $z\in\ell^*$ if $(x_1+sy_1)z_1+(x_2+sy_2)z_2+(x_3+sy_3)z_3=1$ for all $s\in\F_q$. This occurs if $x_1z_1+x_2z_2+x_3z_3=1$ and $y_1z_1+y_2z_2+y_3z_3=0$. These equations are satisfied on a line if $x,y$ are linearly independent, i.e., if $\ell$ does not pass through the origin.

Now there are more than $q^4$ lines $\ell\subset\F_q^3$ which do not pass through the origin. For each line, there is a $1/t!-o(1)$ chance that $|\X\cap\ell|=t$ and, independently, a $1/t!-o(1)$ chance that $|\Y\cap\ell^*|=t$. Each such $\ell$ produces a copy of $K_{t,t}$, so we conclude that the expected number of copies of $K_{t,t}$ in $G$ is at least
\[q^4\paren{\frac1{t!}-o(1)}^2=\Omega_t(q^4)=\Omega_t(n^2),\]
where the last line follows since $n\leq 2tq^2$ deterministically.

Therefore there exists a graph $G$ with the desired properties.
\end{proof}

\section{Subspace evasive sets via randomized algebraic constructions}

In \cite{DL12}, Dvir and Lovett proved that for $c$ sufficiently large in terms of $n,k$, there exists a $(k,c)$-subspace evasive set in $\F_q^n$ of size $\Omega(q^{n-k})$. An alternative proof of this result was given by Sudakov and Tomon where the set $X\subseteq\F_q^n$ is taken to be the common zero set of $k$ random polynomials~\cite[Theorem 3.1]{ST24}.

In the case of $(1,t)$-subspace evasive sets, one can consider $X\subseteq\F_q^n$ to be the zero set of a single random polynomial of degree at most $t$. In this case the analysis is particularly simple and it is not to hard to show, using the methods of~\cite{Buk15,ST24}, that such a set is $(1,t)$-subspace evasive with high probability for $t\geq 4$. Furthermore, using these methods, it is not too hard to show that $X$ has the additional property that we need: for each line, the probability that it contains exactly $t$ points of $X$ is bounded away from zero.

To prove \cref{prop:subspace-evasive} in the case $t=3$, we use a slight modification of this construction which can be analyzed by the same methods.

\begin{proof}[Proof of \cref{prop:subspace-evasive}]
Let $\mathbf f\in\F_q[x_1,x_2,x_3]_{\leq t}$ be a uniform random polynomial of degree at most $t$. Define $\mathbf{X_0}\subseteq\F_q^3$ to be the zero set of $\mathbf f$. 

We first claim that for each line $\ell$, the restriction $\mathbf f|_{\ell}$ is a uniform random univariate polynomial of degree at most $t$. To see this, write $\ell=\alpha+\F_q\beta$ with $\beta\neq 0$. Then there exists some $i\in[3]$ so that $\beta_i\neq 0$; without loss of generality, assume that $i=1$. Let us write
\[\mathbf f(x_1,x_2,x_3)=\sum_{i+j+k\leq t}\mathbf a_{ijk}x_1^ix_2^jx_3^k\]
where the $\mathbf a_{ijk}$ are independent uniform random elements of $\F_q$. To analyze $\mathbf f(\alpha+\beta s)$, first reveal the randomness of $\mathbf a_{ijk}$ for all $i,j,k$ with $(j,k)\neq(0,0)$. Then reveal $\mathbf a_{t00},\ldots,\mathbf a_{000}$ one-by-one. Note that the coefficient of $s^d$ in $\mathbf f(\alpha+\beta s)$ is $a_d+\mathbf a_{d00}\beta_1^d$ where $a_d$ is determined by $\alpha,\beta$ and the coefficients previously revealed. Since $\mathbf a_{d00}$ is a uniform random element of $\F_q$ and $\beta_1\neq 0$, we see that each coefficient of $\mathbf f(\alpha+\beta s)$ is a uniform random element of $\F_q$, independent from all previous coefficients.

In particular, we conclude that $\mathbf f|_{\ell}\equiv 0$ if and only if all $t+1$ of these coefficients are zero, so 
\[\Pr(\mathbf f|_{\ell}\equiv 0)=\frac1{q^{t+1}}.\]

Now fix $t$ distinct collinear points $p_1,\ldots,p_t\in\ell$. Consider the map $\varphi\colon \F_q[x_1,x_2,x_3]_{\leq t}\to\F_q^t$ defined by the evaluations
\[
\varphi(\mathbf f)=\paren{\mathbf f(p_1),\ldots,\mathbf f(p_t)}.
\]
Note that this map is surjective, since by polynomial interpolation there exists a univariate polynomial $\mathbf f|_\ell$ of degree at most $t$ which takes any prescribed values on $t$ distinct points. Such a polynomial can then be extended to a polynomial $\mathbf f$ on $\F_q^3$. Now since $\varphi\colon \F_q[x_1,x_2,x_3]_{\leq t}\twoheadrightarrow\F_q^t$ is a linear map, the preimage of each point has the same size. As $\mathbf f\in\F_q[x_1,x_2,x_3]_{\leq t}$ is uniform, we conclude that the evaluation vector $\varphi(\mathbf f)\in\F_q^t$ is uniform as well.

By linearity of expectation, the above argument implies that the expected number of $t$-tuples of points of $\ell$ on which $\mathbf f$ vanishes is
\[\E\sqb{\binom{|\mathbf{X_0}\cap \ell|}t}=\binom qt q^{-t}=\frac1{t!}-o_{t;q\to\infty}(1).\]

Since $\mathbf f|_\ell$ is a univariate polynomial of degree at most $t$, we know that $|\mathbf{X_0}\cap\ell|\leq t$ unless $\mathbf f|_\ell\equiv 0$. Thus we can write
\[\E\sqb{\binom{|\mathbf{X_0}\cap \ell|}t}=\binom tt\Pr(|\mathbf{X_0}\cap \ell|=t)+\binom qt\Pr(|\mathbf{X_0}\cap \ell|=q)=\Pr(|\mathbf{X_0}\cap \ell|=t)+\binom qt q^{-t-1}.\]
Combining the two equations gives
\[\Pr(|\mathbf{X_0}\cap \ell|=t)=\paren{1-\frac1q}\binom qtq^{-t}=\frac1{t!}-o_{t;q\to\infty}(1).\]

To finish the proof, create $\X\subseteq \mathbf{X_0}$ by removing all points on any line $\ell$ which lies in $\mathbf{X_0}$. Clearly we have $|\X\cap\ell|\leq t$ for every line $\ell$. Call a line $\ell$ \emph{bad} if we removed any point of $\ell$ in producing $\X$ from $\mathbf{X_0}$.

To bound the probability that $\ell$ is bad, note that there are $1+q(q^2+q)=q^3+q^2+1$ lines in $\F_q^3$ which share a point with $\ell$ (including $\ell$ itself). The expected number of these lines on which $\mathbf f$ vanishes identically is thus $(q^3+q^2+1)q^{-t-1}=O(q^{-1})$ since $t\geq 3$. Therefore by Markov's inequality, the probability that $\ell$ is bad is $O(q^{-1})$. Now if $\ell$ is not bad, we have $|\X\cap\ell|=|\mathbf{X_0}\cap\ell|$, so we conclude that there exists an absolute constant $C$ so that
\[\Pr(|\X\cap\ell|=t)\geq\paren{1-\frac 1q}\binom qt q^{-t}-\frac Cq=\frac1{t!}-o_{t;q\to\infty}(1)\]
for each line $\ell$.
\end{proof}

\section{Concluding remarks}

A natural question is if our main result generalizes to $K_{s,t+1}$-free graphs.

\begin{question}
\label{ques:larger-s}
For $3\leq s<t$, is it true that
\[\ex(n,K_{t,t},K_{s,t+1})=\Theta_{s,t}(n^s)?\]
\end{question}

As with the $s=2$ case, the upper bound is immediate. We point out that our construction does not generalize to large values of $s$.

One natural attempt is to consider a point-hyperplane incidence graph in $\F_q^{2s-1}$. In detail, take $X,Y\subseteq\F_q^{2s-1}$ and define $G=(X\sqcup Y,E)$ where $x,y$ are adjacent if $x\cdot y=1$. Let $X,Y\subseteq\F_q^{2s-1}$ be sets so that every $(s-1)$-dimensional affine subspace contains at most $t$ points of $X,Y$ and a positive proportion of these subspaces contain exactly $t$ points of $X,Y$. With these definition, $G$ will have many copies of $K_{t,t}$. Indeed, we will have $|X|,|Y|=\Theta(q^s)$ and $G$ will have $\Theta(q^{s^2})$ copies of $K_{t,t}$.

To see if $G$ is $K_{s,t+1}$-free, we note that if $U\sqcup V$ is a biclique, we must have $U\subseteq H$ and $V\subseteq H^*$ where $H$ is a $d$-dimensional affine subspace and $H^*$ is its dual, a $d^*$-dimensional affine subspace where $d+d^*=2s-2$. Now there will be no copies of $K_{s,t+1}$ where $d=d^*=s-1$ since any $(s-1)$-dimensional affine subspace contains at most $t$ points of $X,Y$. However, the case of $d=s-2$ and $d^*=s$ is problematic. There are many $s$-dimensional affine subspaces which contain at least $t+1$ points of $Y$, so to rule out copies of $K_{s,t+1}$, we need every $(s-2)$-dimensional affine subspace to contain at most $s-1$ points of $X$.

In other words, in addition to being a $(s-1,t)$-subspace evasive, we need that $X,Y$ are also $(s-2,s-1)$-subspace evasive. For $s=2$ this is trivial, but for larger values of $s$ this is a ``general position assumption'', stating that any $s$ points are affinely independent. Sudakov and Tomon showed that such sets must be very small by using the Hamming bound.

For our construction, we need $X,Y$ to have size $\Theta(q^s)$; by \cite[Theorem 1.2]{ST24} there is no such $(s-2,s-1)$-subspace evasive set in $\F_q^{2s-1}$ for $s\geq 4$. To be precise, this result for $s\geq 10$ follows from the stated result, but the bound for all $s\geq 4$ follows easily from their proof. In particular, their proof implies that any $(s-2,s-1)$-subspace evasive set in $\F_q^{2s-1}$ has size at most $O_s(q^{\frac{2s}{\lfloor s/2\rfloor}-1})$.

It remains an interesting open problem to see if this construction can be made to work for $s=3$, or if a different construction can resolve \cref{ques:larger-s} positively for general $s$.

\appendix

\section{\texorpdfstring{F\"uredi's graph contains no copy of $K_{3,t}$}{F\"uredi's graph contains no copy of K3t}}

For completeness, we record the following observation about F\"uredi's construction of an extremal $K_{2,t+1}$-free graph. To define the graph, let $q$ be a prime power with $t\mid (q-1)$, and let $H\le \F_q^\times$ be the multiplicative subgroup of order $t$. The vertex set of F\"uredi's graph $G_t(q)$ is
\[
V(G_t(q))=\paren{\F_q^2\setminus\{(0,0)\}}/H,
\]
where the $H$-action on $\F_q^2\setminus\{(0,0)\}$ is defined by $h(a,b)=(ha,hb)$. The edges are defined by $[a,b]\sim[x,y]$ if and only if $ax+by\in H$. F\"uredi proved that $G_t(q)$ is $K_{2,t+1}$-free, has $n=(q^2-1)/t$ vertices, and every vertex has degree $q$ or $q-1$~\cite{Fur96}. Consequently,
\[
e(G_t(q))=\frac{1}{2}qn+O(n)=\left(\frac{\sqrt t}{2}+o(1)\right)n^{3/2},
\]
matching the classical K\H{o}v\'ari--S\'os--Tur\'an upper bound~\cite{KST54}.

We show that graph $G_t(q)$ does not contain $K_{3,t}$. In fact, for special values of $q$ (specifically $q=t^2+1$), the graph is $K_{3,3}$-free~\cite{Liv21}. Though this result likely fails for general values of $q$, we use similar techniques to prove that $G_t(q)$ is always $K_{3,t}$-free.

In some sense this result is analogous to a result of Ball and Pepe about projective norm graphs. The projective norm graph $H(q,4)$ from \cite{ARS99} was first presented as an example of a dense $K_{4,7}$-free graph; Ball and Pepe~\cite{BP12} later showed that $H(q,4)$ also contains no copies of $K_{5,5}$ (and later generalized this result~\cite{BP16}). Likewise, in our setting, \cref{prop:furedi-no-K3t} shows that F\"uredi's dense $K_{2,t+1}$-free graph also excludes the larger biclique $K_{3,t}$.

\begin{proposition}
\label{prop:furedi-no-K3t}
For $t\geq 2$ and a prime power $q$ with $t\mid q-1$, the graph $G_t(q)$ contains no copy of $K_{3,t}$.
\end{proposition}

\begin{proof}
Assume for contradiction that $G_t(q)$ contains a copy of $K_{3,t}$ with parts $U=\{u,v,z\}$ and $W$. Choose representatives $u=[a,b]$ and $v=[a',b']$. First note that $(a,b)$ and $(a',b')$ are linearly independent over $\F_q$. Indeed, if $(a',b')=c(a,b)$ for some $c\in\F_q^\times$, then $c\not\in H$, since $u\ne v$ are distinct vertices. Let $w=[x,y]$ be adjacent to both $u$ and $u'$, then $ax+by\in H$ and $a'x+b'y=c(ax+by)\in H$. Since $H$ is a multiplicative subgroup, this implies $c\in H$, a contradiction.  Hence $(a,b)$ and $(a',b')$ are linearly independent.

For each $h\in H$, let $(x_h,y_h)\in\F_q^2$ be the unique solution of
\[
ax_h+by_h=1,
\qquad
a'x_h+b'y_h=h,
\]
and set $w_h=[x_h,y_h]$. The existence and uniqueness of $(x_h,y_h)$ follows since $(a,b),(a',b')$ are linearly independent over $\F_q$.

Next we claim that every common neighbor of $u$ and $v$ is one of the vertices $w_h$. Indeed, let $w=[x,y]$ be adjacent to both $u$ and $v$. Then $ax+by\in H$ and $a'x+b'y\in H$. Scaling $(x,y)$ by $(ax+by)^{-1}\in H$, we see that $w=[\tilde x,\tilde y]$ for some $\tilde x,\tilde y$ satisfying
\[
a\tilde x+b\tilde y=1,
\qquad
a'\tilde x+b'\tilde y=h
\]
for some $h\in H$. By uniqueness, $(\tilde x,\tilde y)=(x_h,y_h)$, so $w=w_h$. The vertices $w_h$ are distinct for distinct $h\in H$, so $u$ and $v$ have exactly $|H|=t$ common neighbors. Since $\{u,v,z\}\sqcup W$ is a copy of $K_{3,t}$ in $G_t(q)$, we conclude that $W=\{w_h:h\in H\}$. 

Now write $z=[c,d]$. Since $(a,b)$ and $(a',b')$ form a basis of $\F_q^2$,
there exist scalars $\alpha,\beta\in\F_q$ such that $(c,d)=\alpha(a,b)+\beta(a',b')$. Because $z$ is adjacent to every $w_h$, we have
\[
cx_h+dy_h\in H
\qquad\text{for all }h\in H.
\]
Combing this previous equations, this implies
\[
\alpha(ax_h+by_h)+\beta(a'x_h+b'y_h)=\alpha+\beta h\in H
\qquad\text{for all }h\in H.
\]
Thus $\alpha+\beta h\in H$ for all $h\in H$.

The argument at the beginning of the proof shows that $(a,b),(a',b'),(c,d)$ are pairwise linearly independent over $\F_q$. This implies that  $\alpha,\beta\ne 0$. Since $\beta\ne 0$, the $t$ values $\{\alpha+\beta h:h\in H\}$ are distinct. As they all belong to the $t$-element set $H$, we obtain $\alpha+\beta H=H$. Now $\sum_{h\in H} h = 0$, since $H$ is a nontrivial multiplicative subgroup of $\F_q^\times$. Summing the elements of $\alpha+\beta H=H$ gives
\[
t\alpha+\beta\sum_{h\in H}h=\sum_{h\in H}h=0,
\]
and therefore $t\alpha=0$. Since $t\mid (q-1)$, the characteristic of $\F_q$ does not divide $t$,
so $\alpha=0$, a contradiction. This contradiction shows that $G_t(q)$ contains no copy of $K_{3,t}$.
\end{proof}

\end{document}